\documentclass{article}

\usepackage{amsmath, amsthm, amssymb}
\usepackage{amsfonts}
\usepackage{graphicx}
\usepackage{epic}

\newtheorem{Theorem}{Theorem}
\newtheorem{Corollary}{Corollary}
\newtheorem{Lemma}{Lemma}

\begin{document}

\title{Super FiboCatalan Numbers and their Lucas Analogues}         
\author{Kendra Killpatrick\\ Pepperdine University\\ 24255 Pacific Coast Hwy\\Malibu, CA  90265}

\date{\today}          
\maketitle

\begin{abstract}
Catalan observed in 1874 that the numbers $S(m,n) = \frac{(2m)! (2n)!}{m! n! (m+n)!}$, now called the super Catalan numbers, are integers but there is still no known combinatorial interpretation for them in general, although interpretations have been given for the case $m=2$ and for $S(m, m+s)$ for $0 \leq s \leq 4$.  In this paper, we define the super FiboCatalan numbers $S(m,n)_F = \frac{F_{2m}! F_{2n}!}{F_m! F_n! F_{m+n}!}$ and the generalized FiboCatalan numbers $J_{r,F} \frac{F_{2n}!}{F_n! F_{n+r+1}!}$ where $J_{r,F} = \frac{F_{2r+1}!}{F_r!}$.  In addition, we give Lucas analogues for both of these numbers and use a result of Sagan and Tirrell to prove that the Lucas analogues are polynomials with non-negative integer coefficients which in turn proves that the super FiboCatalan numbers and the generalized FiboCatalan numbers are integers.      
\end{abstract}

\section{Background and Definitions}

The well-known Fibonacci sequence is defined recursively by $F_n = F_{n-1} + F_{n-2}$ with initial conditions $F_0 = 0$ and $F_1 = 1$.  The $n$th Fibonacci number, $F_n$, counts the number of tilings of a strip of length $n-1$ with squares of length 1 and dominoes of length 2.

A second famous sequence, the Catalan sequence,  is defined recursively by $C_n = C_0 C_{n-1} + C_1 C_{n-2} + \cdots + C_{n-2} C_1 + C_{n-1} C_0$ with initial conditions $C_0 = 1$ and $C_1 = 1$.  The Catalan numbers also have an explicit formula given by

\[
C_n = \frac{1}{n+1} \binom{2n}{n}.
\]

Since the Catalan numbers,
\[
\frac{(2n)!}{n! (n+1)!}
\]
are integers, one might wonder if the numbers 
\[
\frac{(2n)!}{n! (n+2)!}
\]
are integers.  Interestingly, these numbers are not necessarily integers but the numbers given by
\[
6\frac{(2n)!}{n!(n+2)!}
\]
do form an integer sequence.  In 1992, Gessel \cite{Ges} showed that, in fact,  the generalized Catalan numbers 
\[
J_r \frac{(2n)!}{n!(n+r+1)!}
\] 
are integers when $J_r$ is chosen to be $(2r+1)!/r!$.

In 2005, Gessel and Xin \cite{GeX} gave a combinatorial interpretation of these numbers for $r=1$ and proved
\[
6\frac{(2n)!}{n!(n+2)!} = 4 C_n - C_{n+1}.
\]

E. Catalan \cite{Cat} observed as far back as 1874 that the numbers
\[
S(m,n)= \frac{(2m)!(2n)!}{m!n!(m+n)!}
\]
are integers, but there is no known combinatorial interpretation for them in general.  Gessel \cite{Ges} called these numbers the {\it{super Catalan numbers}} since $S(1,n)/2$ gives the Catalan number $C_n$.  Note that $S(2,n)/2 = 6\frac{(2n)!}{n! (n+2)!}$.   Allen and Gheorghiciuc \cite{AlG} have given a combinatorial interpretation for $S(m,n)$ in the case $m=2$ and Gheorghiciuc and Orelowitz \cite{GhO} have given a combinatorial interpretation for $T(m,n) = \frac{1}{2} S(m,n)$ for $m=3$ and $m=4$.  Chen and Wang \cite{ChW} have given an interpretation for $S(m, m+s)$ for $0 \leq s \leq 4$.

\subsection{FiboCatalan numbers}

The fibonomial coefficients, an analogue of the binomial coefficients, are defined as 
  
\[
\binom{n}{k}_F = \frac{F_n !}{F_k! F_{n-k}!}
\]
where $F_n! = F_n F_{n-1} \cdots F_2 F_1$.

In 2008, Benjamin and Plott \cite{BeP} gave a combinatorial proof that the fibonomial coefficients are integers, using a notion of staggered tilings.  In 2010, Sagan and Savage \cite{SaS} gave a combinatorial interpretation of the coefficients in terms of tilings associated to paths in a $k$ x $(n-k)$ rectangle.

The FiboCatalan number $C_{n,F}$, first given by Lou Shapiro, is defined as
\[
C_{n,F} = \frac{1}{F_{n+1}} \binom{2n}{n}_F.
\]

Shapiro posed the question about whether these numbers are integers and, if so, whether there is a combinatorial interpretation for them.  The numbers are known to be integers since
\[
C_{n,F} = \binom{2n-1}{n-2}_F + \binom{2n-1}{n-1}_F
\]
but there is still no known combinatorial interpretation for them.

In this paper, we define the {\it{super FiboCatalan numbers}} 
\[
S(m,n)_F= \frac{F_{2m}!F_{2n}!}{F_m!F_n!F_{m+n}!}
\]

and the {\it{generalized FiboCatalan numbers}} as
\[
J_{r,F} \frac{F_{2n}!}{F_n! F_{n+r+1}!}
\]
where $J_{r,F} = F_{2r+1}!/F_r!$.  Note the following relationship between super FiboCatalan numbers and the generalized FiboCatalan numbers:
\[
J_{m-1,F} \frac{F_{2n}!}{F_n! F_{n+m}!} = \frac{F_{2m-1}! F_{2n}!}{F_{m-1}! F_n! F_{n+m}!} = \frac{F_m}{F_{2m}} S(n,m)_F.
\]

\subsection{Lucas analogues}
\vspace{.1in}
The Lucas polynomials $\{ n \}$ are defined in variables $s$ and $t$ as $\{ 0 \} = 0$, $\{ 1 \} = 1$ and for $n \geq 2$ we have $\{ n \} = s \{ n-1 \} + t \{ n-2 \}$.  If $s$ and $t$ are set to be integers then the sequence of numbers given by $\{ n \}$ is called a Lucas sequence.  When $s = t = 1$, the resulting sequence is the Fibonacci sequence.  The lucanomials, an analogue of the binomial coefficients, are then defined as 
\[
\Bigl\{ \begin{array}{c} n\\k \end{array} \Bigr\} = \frac{ \{n\} !}{ \{k\}! \{n-k\}!}
\]
where $\{ n \}! = \{n\} \{n-1\} \cdots \{2\} \{1\}$.  
When $s = t = 1$, $\bigl\{ \begin{array}{c} n\\k \end{array} \bigr\}$ gives the fibonomial coefficients $\binom{n}{k}_F$.

In 2020, Sagan and Tirrell \cite{SaT} gave a new method of proving the lucanomials are polynomials with non-negative integer coefficients by defining a sequence of polynomials, $P_n(s,t)$ called Lucas atoms such that
\[
\{n\} = \prod_{d \vert n} P_d(s,t).
\]

Sagan and Tirrell then prove the following (Theorem 1.1 in their paper):
\begin{Theorem}
Suppose $f(s,t) = \prod_i \{n_i\}$ and $g(s,t) = \prod_j \{k_j\}$ for certain $n_i$, $k_j \in \mathcal{N}$, and write their atomic decompositions as
\[
f(s,t) = \prod_{d \geq 2} P_d(s,t)^{a_d} \ and \ g(s,t) = \prod_{d \geq 2} P_d(s,t)^{b_d}
\]
for certain powers $a_d$, $b_d \in \mathcal{N}$.  Then $f(s,t)/g(s,t)$ is a polynomial if and only if $a_d \geq b_d$ for all $d \geq 2$.  Furthermore, in this case $f(s,t)/g(s,t)$ has nonnegative integer coefficients.
\end{Theorem}

Lucas atoms are irreducible polynomials and have been furthered studied by Alecci, Miska, Murru and Romeo \cite{AMMR}.

The Lucas analogue of the Catalan numbers is given by
\[
C_{ \{n\} } = \frac{1}{ \{n+1\} } \Bigl\{ \begin{array}{c} 2n\\n \end{array} \Bigr\}.
\]

More generally, given two positive integers $a$ and $b$ with $gcd(a,b)=1$, the {\it {rational Catalan number}} is defined as
\[
Cat(a,b) = \frac{1}{a+b} \binom{a+b}{a}.
\]

In this expression, one can set $a=n$ and $b=n+1$ to obtain the usual Catalan numbers.  The Lucas analogue of the rational Catalan numbers is then defined by
\[
Cat \{a,b\} = \frac{1}{\{a+b\}} \Bigl\{ \begin{array}{c} a+b\\a \end{array} \Bigr\}.
\]

The Algebraic Combinatorics Seminar at the Fields Institute \cite{ACS} proved that the q-Fibonacci analogue of Cat(a,b) is a polynomial in $q$ (a method which also works for the Lucas analogue) and in 2020, Sagan and Tirrell \cite{SaT} proved that the Lucas analogue of the rational Catalan numbers is a polynomial with non-negative integer coefficients.  

We can now define the Lucas analogue of the super FiboCatalan numbers as
\[
S \{ m, n \} = \frac{ \{2m\}! \{2n\}!}{ \{m\}! \{n\}! \{m+n\}!}.
\]

and the Lucas analogue of the generalized FiboCatalan numbers as
\[
J_{ \{r\}} \frac{ \{2n\}!}{ \{n\}! \{n+r+1\}!}
\]
where $J_{ \{r\}} = \frac{ \{2r+1\}!}{ \{r\}!}$.

In Section 2 of this paper, we prove that the Lucas analogues of the super FiboCatalan numbers and the generalized FiboCatalan numbers are polynomials with non-negative integer coefficients.  In addition, this gives that the super FiboCatalan numbers and the generalized FiboCatalan numbers are positive integers.  
In Section 3, we give a new identity involving both Fibonacci and FiboCatalan numbers and use it to provide an alternate proof that the generalized FiboCatalan number for $r=1$ is always a positive integer.

\section{The Lucas analogues are polynomials in $\mathbb{N}[s,t]$}

Following the Sagan and Tirrell \cite{SaT} exposition, given a product $f(s,t)$ of Lucas polynomials let
\[
log_d f(s,t) = \text{the power of $P_d(s,t)$ in its Lucas factorization.}
\]

Then Sagan and Tirrell prove the following Lemma (Lemma 3.1 in their paper):
\begin{Lemma}
For $d \geq 2$ we have
\[
log_d \{n\} ! = \lfloor n/d \rfloor.
\]
Furthermore, for integers $m$, $n$, $d$
\[
\lfloor m/d \rfloor + \lfloor n/d \rfloor \leq \lfloor (m+n)/d \rfloor.
\]
\end{Lemma}

Applying this Lemma to the Lucas analogues of the super FiboCatalan numbers we have the following theorems:
\begin{Theorem}
The Lucas analogues of the super FiboCatalan numbers,
\[
S \{ m, n \} = \frac{ \{2m\}! \{2n\}!}{ \{m\}! \{n\}! \{m+n\}!}
\]
are polynomials with non-negative integer coefficients.
\end{Theorem}

\begin{proof}
Applying the previous lemma gives, for any $d \geq 2$,
\[
log_d(\{m\}! \{n\}! \{m+n\}!) = \lfloor m/d \rfloor + \lfloor n/d \rfloor + \lfloor (m+n)/d \rfloor 
\]
and \[
log_d(\{2m\}! \{2n\}!) = \lfloor 2m/d \rfloor + \lfloor 2n/d \rfloor.
\]

By the Division Algorithm, let $m = kd + r$ for $0 \leq r < d$ and $n = ld + s$ for $0 \leq s < d$.  Then we can proceed by cases.

\vspace{.1in}
Case 1:  Let $r < d/2$ and $s < d/2$.  Then $\lfloor m/d \rfloor = k$, $\lfloor n/d \rfloor = l$ and $\lfloor (m+n)/d \rfloor = k+l$.  In this case, $\lfloor 2m/d \rfloor = 2k$ and $\lfloor 2n/d \rfloor = 2l$ thus 
\begin{align*}
log_d(\{m\}! \{n\}! \{m+n\}!) &= \lfloor m/d \rfloor + \lfloor n/d \rfloor + \lfloor (m+n)/d \rfloor \\
 &= k + l + (k+l) = 2k + 2l \\
&= \lfloor 2m/d \rfloor + \lfloor 2n/d \rfloor \\
&= log_d(\{2m\}! \{2n\}!).
\end{align*}

\vspace{.1in}
Case 2:  Without loss of generality, let $d/2 \leq r < d$ and $s < d/2$.  Then $\lfloor m/d \rfloor = k$, $\lfloor n/d \rfloor = l$ and $\lfloor (m+n)/d \rfloor \leq k+l+1$.  In this case, $\lfloor 2m/d \rfloor = 2k + 1$ and $\lfloor 2n/d \rfloor = 2l$ thus 
\begin{align*}
log_d(\{m\}! \{n\}! \{m+n\}!) &= \lfloor m/d \rfloor + \lfloor n/d \rfloor + \lfloor (m+n)/d \rfloor  \\
&\leq k + l + (k+l+1) \\
&= 2k + 1 + 2l \\
&= \lfloor 2m/d \rfloor + \lfloor 2n/d \rfloor \\
&=log_d(\{2m\}! \{2n\}!).
\end{align*}

\vspace{.1in}
Case 3:  Let $d/2 \leq r < d$ and $d/2 \leq s < d$.  Then $\lfloor m/d \rfloor = k$, $\lfloor n/d \rfloor = l$ and $\lfloor (m+n)/d \rfloor = k+l+1$.  In this case, $\lfloor 2m/d \rfloor = 2k + 1$ and $\lfloor 2n/d \rfloor = 2l + 1$ thus 
\begin{align*}
log_d(\{m\}! \{n\}! \{m+n\}!) &= \lfloor m/d \rfloor + \lfloor n/d \rfloor + \lfloor (m+n)/d \rfloor  \\
&= k + l + (k+l+1) \\
&= 2k + 2l + 1 \leq (2k+1) + (2l+1) \\
&= \lfloor 2m/d \rfloor + \lfloor 2n/d \rfloor\\
&=log_d(\{2m\}! \{2n\}!).
\end{align*}

\end{proof}

\begin{Theorem}
The Lucas analogues of the generalized FiboCatalan numbers,
\[
J_{ \{r\}} \frac{ \{2n\}!}{ \{n\}! \{n+r+1\}!}
\]
where $J_{ \{r\}} = \frac{ \{2r+1\}!}{ \{r\}!}$
are polynomials with non-negative integer coefficients.
\end{Theorem}

\begin{proof}
Applying the previous lemma gives, for any $d \geq 2$,
\[
log_d(\{r\}! \{n\}! \{n+r+1\}!) = \lfloor r/d \rfloor + \lfloor n/d \rfloor + \lfloor (n+r+1)/d \rfloor 
\]
and \[
log_d(\{(2r+1)\}! \{2n\}!) = \lfloor (2r+1)/d \rfloor + \lfloor 2n/d \rfloor.
\]

By the Division Algorithm, let $r = kd + t$ for $0 \leq t < d$ and $n = ld + s$ for $0 \leq s < d$.  Then we can again proceed by cases.

\vspace{.1in}
Case 1:  Let $t < d/2$ and $s < d/2$.  Then $\lfloor r/d \rfloor = k$, $\lfloor n/d \rfloor = l$ and $\lfloor (n+r+1)/d \rfloor = k+l+1$.  In this case, $\lfloor (2r+1)/d \rfloor = 2k+1$ and $\lfloor 2n/d \rfloor = 2l$ thus 
\begin{align*}
log_d(\{r\}! \{n\}! \{n+r+1\}!) &= \lfloor r/d \rfloor + \lfloor n/d \rfloor + \lfloor (n+r+1)/d \rfloor \\
 &= k + l + (k+l+1) = 2k + 2l + 1 \\
&= \lfloor (2r+1)/d \rfloor + \lfloor 2n/d \rfloor \\
&= log_d(\{(2r+1)\}! \{2n\}!).
\end{align*}

\vspace{.1in}
Case 2:  Without loss of generality, let $d/2 \leq t < d$ and $s < d/2$.  Then $\lfloor r/d \rfloor = k$, $\lfloor n/d \rfloor = l$ and $\lfloor (n+r+1)/d \rfloor \leq k+l+2$.  In this case, $\lfloor (2r+1)/d \rfloor = 2k + 2$ and $\lfloor 2n/d \rfloor = 2l$ thus 
\begin{align*}
log_d(\{r\}! \{n\}! \{(2r+1)\}!) &= \lfloor r/d \rfloor + \lfloor n/d \rfloor + \lfloor (n+r+1)/d \rfloor  \\
&\leq k + l + (k+l+2) \\
&= 2k + 2l + 2 \\
&= \lfloor (2r+1)/d \rfloor + \lfloor 2n/d \rfloor \\
&=log_d(\{(2r+1)\}! \{2n\}!).
\end{align*}

\vspace{.1in}
Case 3:  Let $d/2 \leq t < d$ and $d/2 \leq s < d$.  Then $\lfloor r/d \rfloor = k$, $\lfloor n/d \rfloor = l$ and $\lfloor (n+r+1)/d \rfloor = k+l+2$.  In this case, $\lfloor (2r+1)/d \rfloor = 2k + 2$ and $\lfloor 2n/d \rfloor = 2l + 1$ thus 
\begin{align*}
log_d(\{r\}! \{n\}! \{(2r+1)\}!) &= \lfloor r/d \rfloor + \lfloor n/d \rfloor + \lfloor (n+r+1)/d \rfloor  \\
&= k + l + (k+l+2) \\
&= 2k + 2l + 2 < (2k+2) + (2l+1) \\
&= \lfloor (2r+1)/d \rfloor + \lfloor 2n/d \rfloor\\
&=log_d(\{(2r+1)\}! \{2n\}!).
\end{align*}

\end{proof}

Since we can recover the super FiboCatalan numbers and the generalized FiboCatalan numbers by setting $s=t=1$ in the Lucas analogues for these numbers, we have the following results for these numbers:

\begin{Corollary}
The super FiboCatalan numbers
\[
S(m,n)_F= \frac{F_{2m}!F_{2n}!}{F_m!F_n!F_{m+n}!}
\]

and the generalized FiboCatalan numbers
\[
J_{r,F} \frac{F_{2n}!}{F_n! F_{n+r+1}!}
\]
where $J_{r,F} = F_{2r+1}!/F_r!$ are positive integers.
\end{Corollary}

\section{An identity and an alternate proof}
 
Some of the special cases of the super FiboCatalan numbers and the generalized FiboCatalan numbers reveal interesting relationships.  For example, when $m=1$, the super FiboCatalan numbers reduce to the FiboCatalan numbers.
\[
S(1,n)_F = \frac{F_2! F_{(2n)}!}{F_1! F_n! F_{(n+1)}!} = \frac{1}{F_{n+1}} \frac{F_{2n}!}{F_n! F_n!} = C_{n,F}
\]

When $m=2$, we have
\[
S(2,n)_F = \frac{F_4! F_{2n}!}{F_2! F_n! F_{(n+2)}!} = \frac{6 F_{2n}!}{F_n! F_{(n+2)}!}.
\]

When $n=m$ we have:
\[
S(m,m)_F = \frac{F_{2m}! F_{2m}!}{F_m! F_m! F_{2m}!} = \binom{2m}{m}_F
\]
and when $n=m+1$ we have
\[
S(m, m+1)_F = \frac{F_{2m}! F_{2m+2}!}{F_m! F_{m+1}! F_{2m+1}!} = \frac{F_{2m+2} F_{2m}!}{F_{m+1}! F_m!} = F_{2m+2} C_{m, F}.
\]

The generalized FiboCatalan number for $r=0$ is equal to $S(1,n)_F$, which is equal to $C_{n,F}$:
\[
J_{0,F} \frac{F_{2n}!}{F_n! F_{n+0+1}!} = \frac{F_1!}{F_0!} \frac{F_{2n}!}{F_n! F_{n+1}!} = C_{n,F} = S(1,n)_F.
\]

The generalized FiboCatalan number for $r=1$ is:
\[
J_{1,F} \frac{F_{2n}!}{F_n! F_{n+1+1}!} = \frac{F_3!}{F_1!} \frac{F_{2n}!}{F_n! F_{n+2}!} = 2 \frac{F_{2n}!}{F_n! F_{n+2}!} = \frac{1}{3} S(2,n)_F.
\]

In this section, we will prove a new identity involving Fibonacci and FiboCatalan numbers and use it to provide an alternate proof that the generalized FiboCatalan numbers for $r=1$ are positive integers.

\begin{Lemma}
\[
F_{2n} F_{n+2} - F_{2n+2} F_n = (-1)^n F_n.
\]
\end{Lemma}

\begin{proof}
This is a fairly well-known result for the Fibonacci numbers and the proof is a straightforward tail-swapping argument similar to those found in Section 1.2, Chapter 1 of Proofs That Really Count by Benjamin and Quinn \cite{BeQ}.  For a more algebraic argument, see Theorem 1.2 (with $q=1$) in a paper by Garrett \cite{Gar}.
\end{proof}

\begin{Theorem}
\begin{equation}\label{mainresult}
F_{2n+1} F_{2n} C_{n,F} - F_{n+1} F_n C_{n+1,F} = (-1)^n F_n F_{2n+1} \frac{F_{2n}!}{F_{n+2}!F_n!} = (-1)^n \binom{2n+1}{n+2}_F.
\end{equation}
\end{Theorem}

\begin{proof}
\begin{align*}
F_{2n+1} F_{2n} C_{n,F} - F_{n+1} F_n C_{n+1,F} &= \frac{F_{2n+1} F_{2n} F_{2n}!}{F_{n+1} F_n! F_n!} - \frac{F_{n+1} F_n F_{2n+2}!}{F_{n+2} F_{n+1}! F_{n+1}!}\\
&=F_{2n+1} F_{2n} F_{n+2} \frac{F_{2n}!}{F_{n+2}! F_n!} - F_{2n+2} F_{2n+1} F_n \frac{F_{2n}!}{F_{n+2}! F_n!}\\
&=F_{2n+1} [F_{2n} F_{n+2} - F_{2n+2} F_n ] \frac{F_{2n}!}{F_{n+2}! F_n!}\\
&=F_{2n+1} (-1)^n F_n \frac{F_{2n}!}{F_{n+2}! F_n!}
\end{align*}
\end{proof}

\begin{Corollary}
For $n \geq 1$, 
\[
F_{2n+1} \frac{F_{2n}!}{F_{n+2}!F_n!} = \frac{1}{F_{n+2}} \binom{2n+1}{n}_F
\]
 is an integer.
\end{Corollary}

\begin{proof}
It is well know that $F_{2n} = F_n F_{n+1} + F_n F_{n-1}$, thus the left side of Equation (\ref{mainresult}) from Theorem 1 is equal to
\[
F_{2n+1} [F_n F_{n+1} + F_n F_{n-1}] C_{n,F} - F_{n+1} F_n C_{n+1,F}
\]
and is therefore divisible by $F_n$.  Using this expression as the left side and dividing both sides of Equation (\ref{mainresult}) by $F_n$ gives
\begin{align*}
F_{2n+1} F_{n+1} C_{n,F} + F_{2n+1} F_{n-1} C_{n,F} - F_{n+1} C_{n+1, F} &= (-1)^n F_{2n+1} \frac{F_{2n}!}{F_{n+2}! F_n!} \\
&= (-1)^n \frac{1}{F_{n+2}} \binom{2n+1}{n}_F.
\end{align*}

Since the left side of this equation is clearly an integer, we have the result.
\end{proof}

Note that the usual binomial expression
\[
\frac{1}{n+2} \binom{2n+1}{n}
\]
is not always an integer since this number is a fraction when $n=2$, for example.

We can also rewrite the expression on the right side of Equation (\ref{mainresult}) as:
\[
(-1)^n F_{2n+1} \frac{F_{2n}!}{F_{n+2}!F_n!} = (-1)^n F_{2n+1} \frac{1}{F_{n+2}} C_{n,F}.
\]

A well-known fact of the Fibonacci numbers is that $gcd(F_n, F_m) = F_{gcd(m,n)}$.  Thus $gcd(F_{2n+1}, F_{n+2}) = F_{gcd(2n+1, n+2)}$.  The $gcd(2n+1, n+2) = 1$ or $3$.  If $gcd(2n+1, n+2)=1$, then $gcd(F_{2n+1}, F_{n+2}) = F_1 = 1$ and so $F_{n+2}$ divides $C_{n, F}$.  If $gcd(2n+1, n+2)=3$, then $gcd(F_{2n+1}, F_{n+2}) = F_3 = 2$ and so $F_{n+2}$ divides $2 C_{n, F}$.

\begin{Corollary}  For $n\geq 1$, the generalized FiboCatalan numbers for $r=1$,
\[
\frac{2 F_{2n}!}{F_{n+2}! F_n!} = \frac{1}{F_{n+2}} 2 C_{n, F}
\]
are positive integers.  
\end{Corollary}

\section{The Lucas analogue of the generalized FiboCatalan numbers}

We have similar relationships and results for the Lucas analogues of the super FiboCatalan numbers and the generalized FiboCatalan numbers.  When $m=1$, the Lucas analogue of the super FiboCatalan numbers reduce to $\{2\}$ times the Lucas analogue of the FiboCatalan numbers.  
\[
S\{1,n\} = \frac{\{2\}! {\{2n\}}!}{\{1\}! \{n\}! \{n+1\}!} = \frac{\{2\}}{\{n+1\}} \frac{\{2n\}!}{\{n\}! \{n\}!} = \{2\} C_{\{n\}}.
\]

When $m=2$, we have
\[
S\{2,n\} = \frac{\{4\}! \{2n\}!}{\{2\}! \{n\}! \{n+2\}!} = \{4\} \{3\} \frac{\{2n\}!}{\{n\}! \{n+2\}!}.
\]
The Lucas analogue of the generalized FiboCatalan number for $r=0$ is equal to $C_{\{n\}}$ which is equal to $\frac{1}{\{2\}} S\{1,n\}$:
\[
J_{\{0\}} \frac{\{2n\}!}{\{n\}! \{n+0+1\}!} = \frac{\{1\}!}{\{0\}!} \frac{\{2n\}!}{\{n\}! \{n+1\}!} = C_{\{n\}} = \frac{1}{\{2\}} \ S\{1,n\}.
\]

When $m=n$ we have:
\[
S\{m,m\} = \frac{\{2m\}! \{2m\}!}{\{m\}! \{m\}! \{2m\}!} = \Bigl\{ \begin{array}{c} 2m\\m \end{array} \Bigr\}
\]
and when $n=m+1$ we have 
\[
S\{m, m+1\} = \frac{\{2m\}! \{2m+2\}!}{\{m\}! \{m+1\}! \{2m+1\}!} = \frac{\{2m+2\} \{2m\}!}{\{m+1\}! \}m\}!} = \{2m+2\} C_{\{m\}}.
\]

The Lucas analogue of the generalized FiboCatalan number for $r=1$ is:
\[
J_{\{1\}} \frac{\{2n\}!}{\{n\}! \{n+1+1\}!} = \frac{\{3\}!}{\{1\}!} \frac{\{2n\}!}{\{n\}! \{n+2\}!} = \{3\}! \frac{\{2n\}!}{\{n\}! \{n+2\}!} = \frac{\{2\}}{\{4\}} S(2,n)_F.
\]

\begin{Lemma}
\[
\{2n\} \{n+2\} - \{2n+2\} \{n\} = (-1)^n \{ 2 \} t^n\{n\}.
\]
\end{Lemma}

\begin{proof}  This proof follows the same tail-swapping argument as the argument for the FiboCatalan case.
\end{proof}

\begin{Theorem}
\begin{align*}
\label{mainresult2}
\{2n+1\} \{2n\} C_{\{n\}} - \{n+1\} \{n\} C_{\{n+1\}} &= (-1)^n t^n \{2\} \{n\} \{2n+1\} \frac{\{2n\}!}{\{n+2\}!\{n\}!} \\
 &= (-1)^n t^n \{ s \} \Bigl\{ \begin{array}{c} 2n+1\\n+2 \end{array} \Bigr\}.
\end{align*}
\end{Theorem}

\begin{proof}
\begin{align*}
\{2n+1\} \{2n\} C_{\{n\}} &- \{n+1\} \{n\} C_{\{n+1\}} \\
&= \frac{\{2n+1\} \{2n\} \{2n\}!}{\{n+1\} \{n\}! \{n\}!} - \frac{\{n+1\} \{n\} \{2n+2\}!}{\{n+2\} \{n+1\}! \{n+1\}!}\\
&=\{2n+1\} \{2n\} \{n+2\} \frac{\{2n\}!}{\{n+2\}! \{n\}!} - \{2n+2\} \{2n+1\} \{n\} \frac{\{2n\}!}{\{n+2\}! \{n\}!}\\
&=\{2n+1\} [\{2n\} \{n+2\} - \{2n+2\} \{n\} ] \frac{\{2n\}!}{\{n+2\}! \{n\}!}\\
&=\{2n+1\} (-1)^n st^n \{n\} \frac{\{2n\}!}{\{n+2\}! \{n\}!}\\
&= (-1)^n t^n \{2n+1\} \{2\} \{n\} \frac{\{2n\}!}{\{n+2\}! \{n\}!}
\end{align*}
\end{proof}

\begin{Corollary}
For $n \geq 1$, 
\[
\{2n+1\} \{2\} \frac{\{2n\}!}{\{n+2\}!\{n\}!} = \{2\} \frac{1}{\{n+2\}} \Bigl\{ \begin{array}{c} 2n+1\\n \end{array} \Bigr\}
\]
 is a polynomial with non-negative integer coefficients.
\end{Corollary}

\begin{proof}
It is well know that $\{2n\} = \{n\} \{n+1\} + t \{n\} \{n-1\}$, thus the left side of Equation (\ref{mainresult2}) from Theorem 1 is equal to
\[
\{2n+1\} [\{n\} \{n+1\} + t \{n\} \{n-1\}] C_{\{n\}} - \{n+1\} \{n\} C_{\{n+1\}}
\]
and is therefore divisible by $\{n\}$.  Using this expression as the left side and dividing both sides of Equation (\ref{mainresult2}) by $\{n\}$ gives
\begin{align*}
\{2n+1\} \{n+1\} C_{\{n\}} + t \{2n+1\} \{n-1\} C_{\{n\}} &- \{n+1\} C_{\{n+1\}} \\
&= (-1)^n t^n \{2n+1\} \{2\} \frac{\{2n\}!}{\{n+2\}! \{n\}!} \\
&= (-1)^n t^n \{2\} \frac{1}{\{n+2\}} \Bigl\{ \begin{array}{c} 2n+1\\n \end{array} \Bigr\}.
\end{align*}

Since the left side of this equation is a polynomial with integer coefficients, we have that
\[
(-1)^n t^n \{2\} \frac{1}{\{n+2\}} \Bigl\{ \begin{array}{c} 2n+1\\n \end{array} \Bigr\}
\]
is a polynomial with integer coefficients and thus
\[
t^n \{2\} \frac{1}{\{n+2\}} \Bigl\{ \begin{array}{c} 2n+1\\n \end{array} \Bigr\}
\]
is a polynomial with non-negative integer coefficients.  The Lucanomial 
\[
\Bigl\{ \begin{array}{c} 2n+1\\n \end{array} \Bigr\}
\]
is a polynomial with non-negative integer coefficients.  Using the facts that $\{ n+2 \}$ can be written as a product of irreducible Lucas atoms \cite{SaT} and that $t^k$ is not a Lucas atom for any $k \geq 1$, we have that 
\[
\{2\} \frac{1}{\{n+2\}} \Bigl\{ \begin{array}{c} 2n+1\\n \end{array} \Bigr\} = \{2n+1\} \frac{1}{\{n+2\}} \{2\} C_{\{n\}}.
\]
is a polynomial with non-negative integer coefficients.
\end{proof}

Writing $\{2n+1\}$ and $\{n+2\}$ in terms of Lucas atoms we have
\[
\{2n+1\} = \prod_{d \vert 2n+1} P_d(s,t) \text{ and } \{n+2\} = \prod_{d \vert n+2} P_d(s,t).
\]

We know $gcd(2n+1, n+2) = 1$ or $3$.  If $gcd(2n+1, n+2) = 1$, then $\{n+2\}$ divides $\{2\} C_{\{n\}}$.  If $gcd(2n+1, n+2)=3$ then $\{n+2\}$ divides $\{3\} \{2\} C_{\{n\}}$, thus
\[
\{3\}! \frac{\{2n\}!}{\{n+2\}! \{n\}!}
\]
is a polynomial with non-negative integer coefficients (i.e. the generalized FiboCatalan number for $r=1$ is a polynomial with non-negative integer coefficients).

\section{$k$-divisible Lucanomials}

Given an integer $k \geq 1$, let 
\[
\{n : k \}! = \{k \} \{ 2k \} \cdots \{ nk \}.
\]
The {\it {$k$-divisible Lucanomial}} is defined as
\[
\left\{ \begin{array}{ccc}  n&:&k \\ m&:&k \end{array} \right\} = \frac{ \{ n:k \}!}{ \{ m:k \}! \{ n-m:k \}!}.
\]
The natural analogue of the super Catalan numbers for the $k$-divisible Lucanomials is then
\[
S\{ m, n : k \} = \frac{ \{ 2m:k \}! \{ 2n:k \} !}{ \{ m:k \}! \{ n:k \}! \{ n+m:k \}!}.
\]

\begin{Theorem}
The $k$-divisible Lucanomial analogues of the super Catalan numbers are polynomials with non-negative integer coefficients.
\end{Theorem}

\begin{proof}
To begin, we note that the Lucas atom $P_d$ divides terms at intervals of length $d/ gcd(d,k)$ in $ \{ n:k \}! $.  Let $M_{d,k} = d/ gcd(d,k)$.  Then 
\[
log_d \{ n:k \}! = \lfloor n/M_{d,k} \rfloor.
\]

The proof now proceeds by using the Division Algorithm and cases as in the proof of the similar result for the lucanomials so we omit the details here.
\end{proof}

\section{Open Problems}
The problem of finding a combinatorial interpretation of the super FiboCatalan numbers remains an interesting open problem, yet will likely prove challenging given that there is a combinatorial interpretation for the super Catalan numbers in only a handful of cases.

In addition, the super Catalan numbers satisfy a number of interesting binomial identities, such as this identity of von Szily (1894), which can be found in \cite{Ges}, Eq. (29), p. 11:
\[
S(m,n) = \sum_{k \in {\mathbb{Z}}} (-1)^k \binom{2m}{m+k} \binom{2n}{n+k}.
\]
Mikic \cite{Mi1} recently proved the following alternating convolution formula for the super Catalan numbers:
\[
\sum_{k=0}^{2n} (-1)^k \binom{2n}{k} S(k,l) S(2n-k,l) = S(n,l) S(n+l,n)
\]
for all non-negative integers $n$ and $l$.  Mikic \cite{Mi2} also proved a similar identity for the Catalan numbers:
\[
\sum_{k=0}^{2n} (-1)^k \binom{2n}{k} C_k C_{2n-k} = C_n \binom{2n}{n}.
\]

We conjecture that many of these identities have analogues for the super FiboCatalan numbers and are interested in exploring these analogues in further research.

\section{Ackknowledgements}  

The author would like to thank Bruce Sagan and the participants of the MSU Combinatorics seminar for their interest in this topic and for their helpful comments on related problems.

\end{document}